 \documentclass[twoside,12pt,reqno]{amsart}
\usepackage{soul,cite,lineno,cancel}
\usepackage{booktabs,multirow,slashbox,xcolor,verbatim,relsize,tikz,tikz-cd,rotating,lscape,soul}
\usepackage{hyperref,enumitem}
\usepackage{latexsym,amsmath,amsthm,amsfonts,amssymb,mathtools}

\newtheorem{thm}{Theorem}[section]
\newtheorem{deff}[thm]{Definition}

\newtheorem{rem}[thm]{Remark}
\newtheorem{prop}[thm]{Proposition}
\newtheorem{cor}[thm]{Corollary}

\newtheorem{que}[thm]{Question}

\newcommand{\vG}{\varGamma}
\newcommand{\ve}{\varepsilon}

\newcommand{\vO}{\varOmega}
\newcommand{\vS}{\varSigma}

\newcommand{\Tau}{\mathcal{T}}

\newcommand{\ov}{\overline}
\newcommand{\mg}{\marginpar}

\def\N{{\mathbb N}}

\def\mcA{\mathcal A}

\def\be{\begin{equation}}
 \newcommand{\dint}{\displaystyle{\int}}
\def\ee{\end{equation}}
\def\ba*{\begin{eqnarray*}}
	\def\ea*{\end{eqnarray*}}
\def\ba{\begin{eqnarray}}
\def\ea{\end{eqnarray}}

\begin{document}

\title[Convergence of measures]{ Convergence for varying measures }
 \subjclass[2020]{Primary 28B20; Secondary 26E25, 26A39, 28B05, 46G10, 54C60, 54C65.}
 \keywords{setwise convergence, convergence in total variation, functional analysis, uniform integrability, absolute integrability,  Pettis integral, multifunction.}
\author{L. Di Piazza, V. Marraffa,
  K. Musia{\l}, A. R. Sambucini}
%================================
\begin{abstract}
Some limit theorems
of the  type
\[\dint_{\Omega}f_n\,dm_n   \rightarrow  \dint_{\Omega}f \,dm\]
 are presented for scalar, (vector), (multi)-valued sequences of $m_n$-integrable functions $f_n$.
The convergences obtained, in the vector and multivalued settings,  are in the weak or in the strong sense. %sense or in the Pettis norm.\\

\end{abstract}
%===============================
\newcommand{\Addresses}{{
  \bigskip
  \footnotesize
\textit{Luisa Di Piazza and Valeria Marraffa}:
 Department of Mathematics, University of Palermo, Via Archirafi 34, 90123 Palermo, (Italy).
 Emails: luisa.dipiazza@unipa.it, valeria.marraffa@\-unipa.it, Orcid ID: 0000-0002-9283-5157,
0000-0003-1439-5501
\\
\textit{Kazimierz Musia{\l}}:
  Institut of Mathematics, Wroc{\l}aw University, Pl. Grunwaldzki  2/4, 50-384 Wroc{\l}aw, (Poland).
  Email: kazimierz.musial@math.uni.wroc.pl, Orcid ID: 0000-0002-6443-2043 \\
\textit{Anna Rita Sambucini%\thanks{ (corresponding author)}
}:
 Department of Mathematics and Computer Sciences, 06123 Perugia, (Italy).
 Email: anna.sambucini@unipg.it, Orcid ID: 0000-0003-0161-8729; ResearcherID: B-6116-2015.
}}
%===============================
\date{\today}
%=======================

\maketitle

\section{Introduction}\label{s-intro}
%\mg{s-intro}

Many problems in measure theory and its  applications deal with sequences of 
measures $(m_n)_n$ converging in some sense rather than with a single  measure $m$.
 Convergence results 
%for  varying measures
have significant applications to various fields  of pure and applied sciences.   Examples
of areas of applications include
% including 
stochastic processes, statistics, control and game theories,   transportation problems, neural networks, 
 signal and image processing
(see, for example, \cite{avendano,F17,F19,danilo,xdanilo,nn1,Pap-,sokol,Fein-arX1902}).

In particular, for the last applications, recently,  multifunctions have been applied   because
the discretization of a continuous signal or image  is affected by
 quantization errors (\hskip-.06cm \cite{mesiar,latorre}) and  its  numerical discretization can be viewed as an
 approximation by means of  a suitable sequence of  multifunctions $(\Gamma_n)_n$
 (as happens in case of scalar functions  \cite{latorre}) which converges
% (in some sense)
to a (multi)-signal $\Gamma$ corresponding to the original  signal.
 Obviously, since  the  signals are
 discontinuous (\hskip-.06cm \cite{MIMMO,carlo})  suitable convergence notions  are needed.
%Our motivation in studying

In the present  paper
we continue the research started in \cite {lerma,lasserre,serfozo,Fein-jmaa,Ma,Fein-arX1807,Fein-arX1902,F17}
for the scalar case
and we  provide sufficient conditions in order to  obtain some kind of Vitali's convergence theorems for a sequence of
 (multi)functions $(f_n)_n$  integrable  with respect to a sequence of measures $(m_n)_n$ . In particular
we  consider   the asympotic properties of $(\int_{\Omega} f_n d m_n)_n$ with respect to setwise
and in total variation  convergences of the measures
 in a arbitrary measurable spaces
$(\vO, \mathcal{A})$. 

The paper is organized as follows: in Section \ref{due}, 
%after the introduction of the setwise and in total variation for varying measures,
we consider the case of the scalar integrands  and 
an analogous of the Vitali's classic convergence result is obtained for  finite  and non negative measures
in Theorem \ref{Th1} using the uniform absolute continuity of the involved integrals
 and the setwise convergence of measures. We compare also our results with the existing ones in literature and
we extend our result to signed measures in Corollary \ref{Th1s}. 
In Section \ref{quattro}
we consider multifunctions taking values in 
 the hyperspace of non empty, weakly compact and convex subsets  of a  Banach space 
and limit results  for the Pettis integral  are  provided both
 in the weak sense (Theorem \ref{Thmulti}) and in the strong sense making use of the Hausdorff metric,
 by means of the convergence in total variation
of measures and the scalar equi-convergence in measure of the sequence of multifunctions (Theorem \ref{Thmulti2}).\\
In Subsection \ref{tre} we get the analogous results when the integrands are vector valued functions.
In this setting   we obtain the  converge both  in the weak sense (Theorem \ref{Th2v}) and in the strong sense
(Theorem \ref{Th1m}).
 Finally in subsection \ref{Mcshane} we consider the McShane integration and in this setting we obtain a convergence result for
 the vector case (Theorem \ref{ThMcSequi}). If $\Gamma_n$'s are McShane integrable multifunction and $i:cb(X)\to \ell^{\infty}(B_{X^*})$ 
is the R{\aa}dstr\"{o}m  embedding, then the vector functions $i \circ \Gamma_n$'s are also McShane integrable and viceversa. So the assumptions for multifunctions
 can be then translated for $i \circ \Gamma_n$'s and we can get the convergence
% for these R{\aa}dstr\"{o}m  images
 from the vector case. 
%Then it should be possible to apply $i^{-1}$ and get the result for multifunctions.
Unfortunately, in general the  R{\aa}dstr\"{o}m  embedding of a Pettis integrable multifunction, even weakly compact valued, is not necessary 
Pettis integrable. So the McShane integrability is essential for this reverse investigation.

%===================================================
\section{The scalar case for integrands}\label{due}

Let   $(\vO,  \mathcal{A}) $ be a measurable space and
let  $\mathcal{M}(\vO)$ br the vector space of {\it finite real-valued measures } on  $(\vO,  \mathcal{A})$. $\mathcal{M}_+(\vO)$
 denotes the cone  of non-negative members of $\mathcal{M}(\vO)$. Let  $|m|$ be the  total variation of a measure $m$.
By the symbol $m \ll \nu$ we denote the usual   absolutely continuity of $m$ with respect to  $\nu$.
%We consider measures without any topological assumption on $\vO$.
We recall that

\begin{itemize}

\item
%\mg{\tiny setwise}
  A sequence $(m_n)_n  \subset  \mathcal{M}(\vO)$ is {\it  setwise convergent} to $m  \in  \mathcal{M}(\vO)$ ($m_n \xrightarrow[]{s}  m$)
	 if  $\lim_n m_n (A) = m(A)$ for every $A \in \mathcal{A}$ (\cite[Section 2.1]{lerma}, \cite[Definition 2.3]{Fein-arX1902}).
\item

A sequence $(m_n)_n  \subset  \mathcal{M}(\vO)$ {\it converges   in total variation} to $m$ ($m_n\stackrel{tv}{\to}m$) if $|m-m_n|(\vO)\to 0$.
Then $(m_n)_n$ is   convergent to $m$ uniformly on $\vS$, (\cite[Section 2]{lasserre}).
\item
A sequence $(m_n)_n  \subset  \mathcal{M}(\vO)$ is bounded if $\sup_n|m_n|(\vO)<\infty$.

\end{itemize}

\begin{rem}\label{rem1}
%\mg{\tiny rem1}
\rm  Since simple functions are dense in the space of bounded measurable functions, a  sequence  $(m_n)_n$ of measures  setwise converges
 to $m$ if and only if 	
\[	\int_{\vO} f d m_n \to \int_{\vO} f d m, \,\,   \ \mbox{ for all  bounded  measurable  } \ \ f:  \vO \to \mathbb R.\]	
	
See also (\cite{Ma}). 	
\end{rem}	
If
$m$ is a     finite  signed measures $m$, we denote by $m^{\pm}$ its  positive and negative parts respectively.
We recall that every finite signed measure has finite total variation.
 Moreover we observe that
\begin{rem}\label{rem2}
%\mg{\tiny rem2}
 \rm
 \rm	Let $(m_n)_n, m$ be measures in $\mathcal{M}(\Omega)$.
If the sequences $(m_n^+)_n$ and $(m_n^-)_n$ are setwise convergent to $m^+$	and $m^-$ respectively, then  $(m_n)_n$
is setwise convergent to $m$. Unfortunately, the reverse implication fails
in general. In fact,
if the reverse implication were valid, then we would have the convergence  $m-m_n\stackrel{s}{\to}0$ and
 hence $(m-m_n)^\pm\stackrel{s}{\to}0$.
 Consequently $|m-m_n|(\vO)\to0$, and this is false.
As a counterexample one can take a sequence $(f_n)_n$ of functions in $L^1(\mu)$ that is weakly convergent to $f\in{L^1(\mu)}$
 but not strongly.
Then take $m_n(E):=\int_E f_n\, d\mu$ and $m(E)=\int_E f\,d\mu$. The convergence $|m-m_n|(\vO)\to0$ means $\int_E |f-f_n|\, d\mu\to 0$,
 which contradicts the assumption that $(f_n)_n$ is not convergent in the norm topology of $L^1(\mu)$.\\
%==========================================================00
%

\end{rem}

\begin{que} What is the relation between convergence in variation and setwise convergence of $(m_n^\pm)_n$ to $m^\pm$?
\end{que}
In general  $(m_n-m)^{\pm} \neq m_n^{\pm} - m^{\pm}$.
Using the Jordan decomposition of a measure and the triangular inequality we have that
\begin{eqnarray*}
m_n^+(E) -m^+(E) &=& \dfrac1{2} \bigg( |m_n|(E) + m_n(E) - |m|(E) - m(E) \bigg)
= \\ &=&
\dfrac1{2} ( |m_n|- |m|)(E) + \dfrac1{2}(m_n  - m)(E) \leq
\\ &\leq&
\dfrac1{2} \biggl| \, |m_n|- |m|\,  \biggr|(E) + \dfrac1{2}(m_n  - m)(E) \leq
\\ &\leq&
\dfrac1{2}  |m_n - m| (E) + \dfrac1{2}(m_n  - m)(E)=
(m_n - m)^+ (E).
\end{eqnarray*}
Moreover
\begin{eqnarray*}
0 \leq \bigl| m_n^+(E) -m^+(E) \bigr| &=&
\dfrac1{2} \biggl| |m_n|(E) + m_n(E) - |m|(E) - m(E) \biggr|
\leq  \\ &\leq&
\dfrac1{2} \biggl| |m_n|(E) - |m|(E) \biggr|  + \dfrac1{2}\biggl| m_n(E)  - m(E) \biggr|\leq
\\ &\leq&
\dfrac1{2} \biggl| \, |m_n|- |m|\,  \biggr|(E) + \dfrac1{2}|m_n  - m|(E) \leq
\\ &\leq&
\dfrac1{2}  |m_n - m| (E) + \dfrac1{2}|m_n  - m|(E)=
|m_n - m| (E).
\end{eqnarray*}
Analogously we could prove the inequality with $(m_n-m)^-$ and so the convergence in (total) variation is stronger that the setwise
convergence of $(m_n^\pm)_n$ to $m^\pm$.

\begin{rem}\label{nota}
%\mg{\tiny nota}
\rm
 Observe that if  $\nu_n = m_n - m$,\, with $m_n, m \in \mathcal{M}_+(\Omega)$ for every $n \in \mathbb{N}$,  then
$\nu_n^+ \leq m_m$ and $\nu_n^- \leq m$. In fact, if $(P_n, N_n)$ is a Hahn decomposition for $\nu_n$ then, for every $E \in \mcA$ it is
%\mg{\tiny nu}
\begin{eqnarray}\label{nu}
\nu_n^+(E) &=& \nu_n(E \cap P_n) = m_m(E \cap P_n) - m(E \cap P_n) \leq
\\ &\leq& m_n(E \cap P_n) \leq m_n(E); \nonumber \\
\nu_n^-(E) &=& - \nu_n (E \cap N_n) = m(E \cap N_n) - m_n(E \cap N_n) \leq \nonumber \\
&\leq& m(E \cap N_n) \leq m(E). \nonumber
\end{eqnarray}
If $f$ is a non-negative function integrable with respect $m$ and $m_n$ for every $n \in \mathbb{N}$ then, for every $E \in \mcA$
\begin{eqnarray*}
\int_E f d\nu_n^+ \leq \int_E f dm_n, \qquad \int_E f d\nu_n^- \leq \int_E f dm.
\end{eqnarray*}
(This is true for simple functions $s$ and then we apply with $0 \leq s \leq f$.)
So,  if $f \in L^1(m_n) \cap L^1(m)$, then $f \in L^1(|\nu_n|)$.
Moreover, since
\begin{eqnarray*}
\int_E f d(m_n -m) = \int_E f dm_n - \int_E f dm,
\end{eqnarray*}
we have that, by formulas (\ref{nu})
%\mg{\tiny \vskip1cm variaz}
\begin{eqnarray}\label{variaz}
 \left|\int_E f d\nu_n\right| &=&
% \left| \int_E f d(m_n -m) \right|  =
\left| \int_E f dm_n - \int_E f dm \right|
= \left|\int_E f d\nu_n^+
- \int_E f d\nu_n^-
\right| \\
%\quad \\
%&=& \nonumber
% \left|\int_E f^+ d\nu_n^+ -  \int_E f^+ d\nu_n^-
%- \int_E f^- d\nu_n^+  +  \int_E f^- d\nu_n^- \right|
%\\
&\leq& \nonumber
\int_E |f| d\nu_n^+ + \int_E |f| d\nu_n^- = \int_E |f| d|\nu_n|  \leq
\\ &\leq& \nonumber
  \int_E |f| dm_n +  \int_E |f| dm < +\infty.
\end{eqnarray}
\phantom{a} \hfill $\Box$
\end{rem}

%\vskip.3cm
 From now on we assume that all the measures we consider belong to $\mathcal{M}_+(\Omega)$ unless otherwise specified.\\

Let's start by proving an analogue of Vitali's classic theorem for  varying measures.
We can have different versions of it depending on the assumptions we use on $(f_n)_n$ and on the varying measures  $(m_n)_n$.
\begin{deff}
	\rm Let  $(m_n)_n$ be  a sequence  of measures.
 We say that:
\begin{itemize}	
\item[\rm (u.a.c.)]
a  sequence of measurable functions $(f_n)_n: \vO \to \mathbb{R}$
		has {\it  uniformly absolutely   continuous  $(m_n)$-integrals   on $\vO$},  if for every $ \varepsilon  >0$ there exists $\delta >0$ such that 	
	 for every $n \in \mathbb{N}$			
%\mg{\tiny Pettis1}
	\begin{equation}\label{Pettis1}
			\left(A \in  {\mathcal A} \quad \mbox{and} \ \  m_n(A)<\delta \right)  \ \ \Longrightarrow \ \ \int_A|f_n|\,dm_n < \varepsilon.
	\end{equation}	
If $f_n=f$ for all $n\in\N$, then we say that $f$ has {\it  uniformly absolutely   continuous  $(m_n)$-integrals   on $\vO$.}\\
If $m_n=m$ for all $n\in\N$, then we say that $f_n's$ have {\it  uniformly absolutely   continuous  $m$-integrals   on $\vO$.}

\item[\rm (u.i.)]  a  sequence of measurable functions $(f_n)_n: \vO \to \mathbb{R}$
	is  {\it  uniformly $(m_n)$-integrable on $\vO$},  if
%\mg{\tiny UI}
\begin{eqnarray}\label{UI}
\lim_{\alpha  \to +\infty } \sup_n \int_{\{|f_n| >\alpha\}} |f_n|\, dm_n \  =0.
\end{eqnarray}
If $f_n=f$ for all $n\in\N$, then we say that $f$ is  {\it  uniformly $(m_n)$-integrable on $\vO$.}\\
If $m_n=m$ for all $n\in\N$, then we say that $f_n's$ are  {\it  uniformly   $m$-integrable  on $\vO$.}
\end{itemize}	
\end{deff}
It is obvious that if a sequence $(f_n)_n$ of measurable fuctions is uniformly bounded, then it is uniformly $(m_n)_n$-integrable
 for an arbitrary $(m_n)_n$ such that $\sup_n m_n(\vO)<+\infty$.\\

 The following result is contained in Serfozo's paper  \cite[Lemma 2.5 (i) $\Longleftrightarrow$ (iii)]{serfozo} where the
author gives a different proof using Markov's inequality
and the tight $(m_n)_n$-integrability condition of $(f_n)_n$. Moreover,
 in \cite{serfozo}, the uniform absolute continuity is given in a  slightly different form,
but Serfozo's and our definitions are equivalent.

%\mg{\tiny p4}
\begin{prop}\label{p4}
 Let  $(m_n)_n$  be  a bounded sequence of  measures
and 	 $(f_n)_n$ be
 a sequence of real valued measurable functions on $\vO$. Then, the sequence $(f_n)_n$
	is   uniformly $(m_n)$-integrable on $\vO$ if and only if it has uniformly absolutely continuous $(m_n)$-integrals and
%\mg{\tiny e12}
\begin{equation}\label{e12}
\sup_n \int_{\Omega} |f_n|\, dm_n < +\infty\,.
\end{equation}
\end{prop}
\begin{proof}
Assume that $(f_n)_n$ is uniformly $(m_n)$-integrable on $\vO$. Let $\ve>0$ be fixed and $\alpha_0 > 0$ be such that
 $\sup_n\int_{\{|f_n|>\alpha_0\}}|f_n|\,dm_n<\ve/2$. Let $\delta=\ve/{2\alpha_0}$. If $m_n(A)<\delta$, then

\begin{eqnarray*}
\int_A|f_n|\,dm_n&=&\int_{A\cap\{|f_n|\leq\alpha_0\}}|f_n|\,dm_n+\int_{A\cap\{|f_n|>\alpha_0\}}|f_n|\,dm_n\\
&\leq&\alpha_0m_n(A)+\ve/2<\ve\,.
\end{eqnarray*}
Let
$K:= \sup_n m_n(\Omega) < +\infty$.
The inequality  \eqref{e12} is a consequence of
\[
\sup_n\int_{\vO}|f_n|\,dm_n\leq \alpha_0\sup_n m_n(\{|f_n|\leq\alpha_0\})+\ve/2
\leq  K \alpha_0 + \ve/2.
\]
Assume now the uniform absolute continuity of $(f_n)_n$ and the validity of \eqref{e12}. Let $a:=\sup_n \int_{\Omega} |f_n|\, dm_n < +\infty$.
 We have
\[
a  \geq  \sup_n\int_{\{|f_n|>\alpha\}}|f_n|\,dm_n\geq\alpha\sup_nm_n\{|f_n|>\alpha\}
\]
and so $\lim_{\alpha\to\infty}\sup_n m_n(\{|f_n|>\alpha\})=0$.
 If $\ve$ and $\delta$ are as in \eqref{Pettis1}, then there exists $\alpha_0$ such that
for all $\alpha>\alpha_0$
$
 \sup_n m_n(\{|f_n|>\alpha\})<\delta\,.
$\\
Since the sequence has uniformly absolutely continuous $(m_n)$-integrals, for all $\alpha>\alpha_0$ we get the inequality
\[
  \sup_n \int_{\{|f_n|>\alpha\}} |f_n|\,dm_m<\ve\]
and that proves the required uniform $(m_n)$-integrability.
\end{proof}

Before proceeding we would observe
that the boundedness of  $(m_n)_n$ is used only in the if part of the previous proof.
Moreover we want to highlight that
%our assumption $\sup_n m_n(\Omega) < +\infty$
it  implies the tightly $(m_n)_n$-integrability
 of the sequence $(f_n)_n$ as observed in \cite[Formula (2.5) pag 283]{serfozo}.\\

As a consequence of the previous result we obtain
%\mg{\tiny c2}
\begin{cor}\label{c2}
 Let $(m_n)_n$  be a bounded sequence of  measures.
A measurable function $f: \vO \to \mathbb{R}$  is
uniformly $(m_n)$-integrable on $\vO$ if and only if
 $f$ has uniformly absolutely continuous $(m_n)$-integrals and
%\mg{\tiny e4}
\begin{equation}\label{e4}
\sup_n \int_{\Omega} |f|\, dm_n < +\infty\,.
\end{equation}
\end{cor}

%\mg{\tiny c1}
\begin{cor}\label{c1}
If the sequence $(m_n)_n$
%in ${\color{blue}\mathcal{M}_+(\vO)}$
 is setwise  convergent to $m$, then  a measurable function $f: \vO \to \mathbb{R}$
 has uniformly absolutely continuous $(m_n)$-integrals if and only if $f$  is uniformly $(m_n)$-integrable on $\vO$. Moreover,
%\mg{\tiny e6}
\begin{equation}\label{e6}
\sup_n \int_{\Omega} |f|\, dm_n < +\infty\,.
\end{equation}
\end{cor}
\begin{proof}
Assume that  $f$ has uniformly absolutely continuous $(m_n)$-integrals and $(m_n)_n$ is setwise convergent to $m$.
Given $\ve>0$ let $\delta>0$ be such that \eqref{Pettis1} is fulfilled.
For each $a >0$ let $E_{a}:= \{t \in \Omega : |f(t) |>a\}$.
 There exists $a_0$ with
$m (E_{a_0}) <\delta/2.$
 In fact if we suppose by absurd that such  $a_0 $ does not exist, then we can construct a sequence of positive numbers $(b_n)_n \uparrow +\infty$  for which
 $m(\bigcap_n E_{b_n}) = \lim_n m(E_{b_n}) \geq \delta/2$, since $E_{b_{n+1}} \subseteq E_{b_n}$ for every $n \in \mathbb{N}$.\\
So the set $\bigcap_n E_{b_n} = \{t  \in \vO: f(t) = +\infty\}$ has positive measure
which is in contradiction with the hypothesis $f: \vO \to \mathbb{R}$.
Moreover, due to the   setwise  convergence of the measures,  there exists $n_0 (a_0) \in\N$ such that $\sup_{n\geq{n_0}}m_n(E_{a_0})<\delta/2$.\\
Then, analogously,  for each $i \leq{n_0}$, there exists $a_i>0$ with $m_i(E_{a_i})<\delta/2$.
 Let $\ov{a}=\max\{a_1,\ldots,a_{n_0}, a_0\}$.\\
So, by the monotonicity,  $\sup_n\int_{E_a}|f|\,dm_n<\ve$ for every $ a \geq \ov{a}$. The inequality \eqref{e6} follows from Corollary \ref{c2}.
\end{proof}

\begin{rem} \rm
If we assume that $f$ is $ \ov{\mathbb{R}}$ valued then
	we obtain the inequality (\ref{e6})
% second thesis
 of Corollary \ref{c1} under the additional hypothesis $f \in L^1(m)$, or $m\{|f|=+\infty\}=0$.
\end{rem}
The first part of the following result is contained  in  \cite[Lemmata 2.2 and 2.5,  Theorem 2,4]{serfozo}.
 For  the convenience of the reader we prefer to give here a direct proof.

%\mg{\tiny \vskip.5cm p1}
\begin{prop}\label{p1}
 Let $m$ and $(m_n)_n$
% $n \in\N$,
 be   measures
such that the sequence $(m_n)_n$ is setwise  convergent to $m$.
Moreover let $f: \vO \to \mathbb{R}$
 have uniformly  absolutely continuous $(m_n)$-integrals on $\varOmega$.
Then  $f \in L^1(m)$ and for all $A \in {\mathcal A}$	
\begin{equation}\label{convf}
	\lim_n \int_{A} f\,dm_n =  \int_{A} f\,dm.
\end{equation}
\end{prop}
\begin{proof}
 By Corollary \ref{c1}  we have that $\sup_n \int_{\Omega} |f| dm_n < +\infty\,$.
The equality (\ref{convf}) holds for simple functions.
So if $0\leq{f_k}\nearrow|f|$ are simple, then
\begin{eqnarray}\label{e3}
\int_{\Omega}f_k dm \leq \liminf_n  \int_{\Omega} |f| dm_n \leq \sup_n \int_{\Omega} |f| dm_n \stackrel{Cor. \ref{c1}}{<} +\infty\,.
\end{eqnarray}
Now we apply the Lebesgue Monotone Convergence Theorem and  obtain that $f \in L^1(m)$.\\
We will show now that (\ref{convf}) holds. It is sufficient to prove it for $\vO$.
Let   $\varepsilon > 0$ be fixed. By the hypothesis and Corollary \ref{c1}  there exists $\alpha_{\varepsilon}$ such that
%\mg{\tiny formula}
\begin{eqnarray}\label{formula}
 \sup \left\{ \int_{\{|f| >  \alpha_{\varepsilon}\}} |f| dm , \int_{\{|f| >  \alpha_{\varepsilon}\}} |f| dm_n,\,\, n \in \mathbb{N} \right\} < \varepsilon.
\end{eqnarray}

Moreover by the classical Dominated Convergence Theorem for varying measures (see e.g. \cite{royden} Ch.11 Proposition 18)
 if $A\in \mathcal A$ is fixed, then there exists $n_0$ such that for every $n>n_0$
%\mg{\tiny formula1}
\begin{eqnarray}\label{formula1}
\left| \int_{\{|f| \leq  \alpha_{\varepsilon}\}} f\,dm - \int_{\{|f| \leq  \alpha_{\varepsilon}\}} f\, dm_n \right| < \varepsilon.
\end{eqnarray}	
	
	Then by  \eqref{formula} and \eqref{formula1}  for  $n>n_0$ we have
\begin{eqnarray*}
	\left|\int_{\Omega} f\,dm_n - \int_{\Omega} f\,dm\right| &\leq &
	\left| \int_{\{|f|
	\leq  \alpha_{\varepsilon}\}} f dm_n - \int_{\{|f| \leq  \alpha_{\varepsilon}\}} f dm \right|+\\
	 &+& \left|\int_{\{|f| >  \alpha_{\varepsilon}\}} f dm\right| +
	 \left|\int_{\{|f| >  \alpha_{\varepsilon}\}} f dm_n\right|<	3\varepsilon,
  \end{eqnarray*}
and the thesis follows.

\end{proof}

 The next convergence  result is obtained also in  \cite[Theorems 2.7 and  2.8]{serfozo} using the thightness and the uniform 
$(m_n)_n$-integrability of $f$ and $(f_n)_n$,
and later extended in \cite[Corollary 5.3]{Fein-arX1902}, in which the above hyphoteses on $f$ are omitted.
We   give here a proof involving the uniform absolute continuity.

%=========================
\begin{thm}\label{Th1}
%\mg{\tiny Th1}
	  Let $f, f_n: \vO \to \mathbb{R}$  be  measurable functions and let $m$ and $(m_n)_n$,   be  measures.
% in ${\color{blue}\mathcal{M}-+(\vO)}$.
	 Suppose that
	 \begin{itemize} 	
	 	\item[\rm \rm(\ref{Th1}.i )]  $(f_n)_n$  has uniformly absolutely   continuous  $(m_n)$-integrals   on $\vO$; 	
	 	\item[\rm \rm(\ref{Th1}.ii)]  $f_n(t)  \rightarrow f(t)$,  in $m$-measure, as $n \rightarrow 		\infty$;
	 	\item[\rm \rm(\ref{Th1}.iii)] $f $  has uniformly absolutely   continuous  $(m_n)$-integrals     on $\vO$;
	 	\item[\rm \rm(\ref{Th1}.iv)] $(m_n)_n$ is setwise  convergent to $m$.
	 \end{itemize} 	
 	 Then, for all $A \in {\mathcal A}$, 	
%\mg{\tiny Th1.conv}
	\begin{equation}\label{Th1.conv}
		 \lim_n \int_{A} f_ndm_n =  \int_{A} fdm.
	\end{equation}	
\end{thm}
%================
\begin{proof} 
From (\ref{Th1}.ii) there exists a subsequence of $(f_n)_n$ which converges $m$-a.e.
to $f$, for simplicity we denote  it again  $(f_n)_n$.
It is sufficient to prove the equality (\ref{Th1.conv}) for $A= \vO$. 	
By Proposition \ref{p1}  $f\in L^1(m)$. Fix $\varepsilon >0$.
Let
$\delta:= \min\left\{ \varepsilon, \delta(\varepsilon/6) , \delta_f(\varepsilon/6)\right\} >0$,
where $\delta(\varepsilon/6)$ satisfies (\ref{Pettis1}) with $\varepsilon/6$ and
$\delta_f(\varepsilon/6)$ is that of the absolute continuity of $\int f dm$ corresponding again to  $\varepsilon/6$.\\
By  (\ref{Th1}.ii) and  the Egoroff's Theorem, we can find a set $E \in  {\mathcal A}$ such that $f_n \rightarrow f$ uniformly on $E^c$
 and $m(E)<  \delta/2$.\\
Taking into account  (\ref{Th1}.iv) let  now $N_0 \in\N$ be such that
%\mg{\tiny Pettis200}
\begin{equation}\label{Pettis200}
|m_n(E)-m(E)| < \frac{ \delta }{2} \ \qquad\mbox{and} \ \  |m_n(E^c)-m(E^c)| < 1,
\end{equation}
for every $n>N_0$.
Moreover, since the convergence is uniform on $E^c$, let   $N_1 \in\N$ be such that

%\mg{\tiny 800}
\begin{equation}\label{800}
	|f_n(t)-f(t)| < \frac{ \varepsilon }{6} \cdot \frac{ 1 }{m(E^c)+1}  ,
\end{equation}
for every $t \in E^c$ and for every $n>N_1$. Then, for every $n>N_1$,
%\mg{\tiny 1100}
\begin{equation}\label{1100}
	\int_{E^c}|f_n-f|dm < \frac{ \varepsilon }{6}.
\end{equation}

Therefore by (\ref{Pettis200}) and (\ref{800}) we obtain, for every for $n>\max \{N_0, N_1\}$,
%\mg{\tiny 900}
\begin{equation}\label{900}
	\int_{E^c}|f_n-f|dm_n< \frac{ \varepsilon }{6} \cdot \frac{m_n(E^c )}{m(E^c)+1}
	 < \frac{ \varepsilon }{6} \cdot \frac{ m(E^c)+1}{m(E^c)+1}   =\frac{ \varepsilon }{6}.
\end{equation}
Since,  $m(E) < \delta/2$, by (\ref{Pettis200}) we have also
$m_n(E) <  \delta$
for every for $n> N_0$. Then  by (\ref{Th1}.i), for every $n> N_0$  we get
%\mg{\tiny 1000}
	\begin{equation}\label{1000}
	\sup \left\{ \int_E|f_n|dm_n ,  \,\, \int_E|f|\,dm \right\} <\frac{\ve}{6}\,.
	\end{equation}
Now taking into account  hypothesis (\ref{Th1}.iv) and Proposition \ref{p1},
 let $N_2$ be such that
%\mg{\tiny 1400}
	\begin{equation}\label{1400}
\left| \int_{\vO} f dm - \int_{\vO} f dm_n  \right| < \frac{ \varepsilon}{2}
\end{equation}
for  every for $n> N_2$.
 Therefore by   (\ref{1100}--\ref{1400}),  for $n>\max \{N_0, N_1, N_2\}$ we infer
 \begin{eqnarray*}
&& \left| \int_{\vO} f dm -   \int_{\vO} f_n dm_n \right| \leq
	 \left| \int_{\vO} (f_n -f) dm_n \right| +   \left| \int_{\vO} f dm - \int_{\vO} f dm_n  \right| \\
	&\leq& \left| \int_{E^c} (f_n -f) dm_n \right| +  \int_{E} |f_n| dm_n + \int_{E} |f| dm_n +
 \left| \int_{\vO} f dm - \int_{\vO} f dm_n  \right|\\
	&\leq&    \int_{E^c} |f_n -f| dm_n   +  \int_{E} |f_n| dm_n +
	 \int_{E} |f| dm + \left| \int_{\vO} f dm - \int_{\vO} f dm_n  \right|\\
 	&< &  \frac{ \varepsilon}{2}+\left| \int_{\vO} f dm - \int_{\vO} f dm_n  \right|
	< \frac{ \varepsilon}{2}  +\frac{ \varepsilon}{2}=\varepsilon.
 \end{eqnarray*}
 This implies that  equality (\ref{Th1.conv}) is valid for the initial sequence because if, absurdly, a subsequence
 existed in which it is not valid, there would be a contradiction.
\end{proof}

%\Kd

\begin{rem}\label{quest}
%\mg{\tiny quest}
\rm
In light of the quoted result in \cite{Fein-arX1902}  the question  if the setwise convergence of $m_n$ to $m$, the
convergence  in $m$-measure of $f_n$ to $f$ and the uniformly absolutely   continuous  $(m_n)$-integrability
  on $\varOmega$ of $f_n$ permit to obtain the  uniformly absolutely   continuous  $(m_n)$-integrability
   on $\varOmega$ of the function $f$ arises spontaneously. \\
A partial positive answer can be given
at least under the additional hypothesis that $m_n \leq m$ for each $n$ (this means that the measures $m_n$
 are dominated by $m$ but no monotonicity is required to the sequence $(m_n)_n)$.		
Indeed let assume  conditions (\ref{Th1}.i), (\ref{Th1}.ii),  (\ref{Th1}.iv) and $m_n \leq m$ for each $n$, then
fix $\varepsilon >0$, let
$\delta>0$ be such that (\ref{Pettis1}) is satisfied and let  $A \in  {\mathcal A}$  with   $m_n(A)<\delta$.
By the Fatou Lemma for converging measures (see \cite[p.231]{royden}) we have
	\begin{eqnarray*}
 \int_A|f|\,dm_n \leq \int_A|f|\,dm \leq \liminf_n  \int_A|f_n|\,dm_n < \varepsilon.
\end{eqnarray*}	
So  $f$  has uniformly absolutely   continuous  $(m_n)$-integrals and taking into account  Proposition 2.4,
 condition (\ref{Th1}.iii) follows.\\
The same holds if $f$ is measurable and bounded thanks to the setwise convergence. What happens if $f$ is unbounded  is unknown.\\
Comparing the Feinberg-Kasyanov-Liang  result  \cite[Corollary 5.3]{Fein-arX1902}
 with Theorem \ref{Th1} we can observe also that  the hypotheses
assumed in the quoted paper imply that
\mbox{$\sup_n \int_{\Omega} |f_n| dm_n <+\infty$} which is not assumed in our theorem; instead of it  we require the condition (\ref{Th1}.iii).
\end{rem}

If we consider signed measures in $\mathcal{M}(\vO)$ we get

\begin{cor}\label{Th1s}
%\mg{\tiny \vskip.2cm Th1s}
	  Let $f, f_n: \vO \to \mathbb{R}$  be  measurable functions and let $m$ and $(m_n)_n$,  be   measures on  $\mathcal{M}(\vO)$.
	 Suppose that
	 \begin{itemize} 	
	 	\item[\rm \rm (\ref{Th1s}.i)]  $(f_n)_n$  has uniformly absolutely   continuous  $(|m_n|)$-integrals   on $\vO$; 	
	 	\item[\rm \rm (\ref{Th1s}.ii)]  $f_n(t)  \rightarrow f(t)$, in $|m|$-measure as $n \rightarrow \infty$;
	 	\item[\rm \rm (\ref{Th1s}.iii)] $f $ is uniformly $(m_n^\pm)$-integrable on $\vO$;
	 	\item[\rm \rm (\ref{Th1s}.iv)] $(m_n^\pm)_n$ is setwise  convergent to $m^\pm$.
	 \end{itemize} 	
 	 Then, for all $A \in {\mathcal A}$, 	
	\begin{equation}\label{Th1s.conv}
		 \lim_n \int_{A} f_ndm_n =  \int_{A} fdm.
	\end{equation}	
\end{cor}
\begin{proof}
It is enough to apply Remark \ref{rem2} and Theorem \ref{Th1} to the pair $(m_n^{\pm},m^{\pm})$.
\end{proof}

%======================capitolo multivoco==============================
\section{The multivalued case for integrands}\label{quattro}
Let $X$ be a Banach space with dual $X^*$ and let 
 $B_{X^*}$ be the unit ball of $X^*$.
The symbol $c(X)$ stands for the collection of all
nonempty closed convex subsets of $X$ and $cwk(X)$ (resp. $cb(X)$) denotes
the family of all weakly compact  (resp. bounded) members of $c(X)$.  For every $C \in c(X)$ the
{\it  support function of}   $\, C$ is denoted by $s( \cdot, C)$ and
defined on $X^*$ by $s(x^*, C) = \sup \{ \langle x^*,x \rangle \colon  \  x \in C\}$, for each $x^*$. If $C, D \in cwk(X)$
 (resp. $cb(X)$)
 then
$d_H(C,D):= \sup_{||x^*||\leq1}|s(x^*,C) -s(x^*,D)|$ is the Hausdorff metric on the hyperspace $cwk(X)$ (resp. $cb(X)$).\\
%Now our aim is to  consider an analogue of Theorem {\color{blue}  \ref{Th2v} } for multifunctions.

Let us recall some fact on multifunctions.
Any map $\vG:\vO\to c(X)$ is called a {\it multifunction}.
 A multifunction $\Gamma$ is said to be {\it scalarly measurable} if for every $x^*\in X^*$, the function $t\to
s(x^*,\Gamma(t))$ is measurable;  $\Gamma$  is said to be {\it scalarly integrable} if for each $x^*\in X^*$, the  function $t\to
s(x^*,\Gamma(t))$ is integrable.

The multifunction $\Gamma$ is said to be  {\it Pettis integrable }   in $cwk(X)$   with respect to  a measure  $m$
%\in  {\mathcal M}_{+}(\vO)$
if
 $\Gamma$ is scalarly integrable with respect to $m$ and for every
 $A \in  {\mathcal A}$, there exists $M_{\Gamma}(A) \in cwk(X)$ such that
\[
s(x^*,M_{\Gamma}(A))=\int_A s(x^*,\Gamma)\, dm \; {\rm for} \; {\rm all }   \;     x^*\in X^*.
\]
We set $ \int_A  \Gamma\, dm:= M_{\Gamma}(A)$.
For what concerns the multivalued integrability in $cwk(X)$ we refer, for example, to \cite{mimmo1,mimmo2,dm2020,M2}.
If $\Gamma$ is single-valued we obtain the  well known  definition for the vector functions.
%; in particular
% we define the Pettis norm  $\| \cdot \|_{m,Pe}$ of a  function $f:\Omega \to X$  as follows:
%\[ \| f\|_{m,Pe} := \sup_{\|x^*\| \leq 1 } \int_{\vO} | x^*f| dm. \]
As regards Pettis integrability  we refere to \cite{M,M2,mu8,pallares}.

\begin{deff}\label{def3.1}
Let  $(m_n)_n$ be  a sequence of   measures.
%in  ${\mathcal M}^+(\vO)$.
Moreover,  for each $n \in\N$, let $\Gamma_n :\vO \rightarrow cwk(X)$ be a multifunction  scalarly integrable with respect to $m_n$.
We say that the sequence $(\Gamma_n)_n$
 has uniformly absolutely   continuous scalar $(m_n)$-integrals   on
$\vO$  if for every $ \varepsilon  >0$
there exists $\delta >0$ such that for every $n\in\N$ and $A \in  {\mathcal A}$
%\mg{\tiny formula-def3.1}
\begin{equation}\label{formula-def3.1}
\hskip-.4cm	
 m_n(A)<\delta  \ \Rightarrow \
	\sup\left\{\int_A|s(x^*, \Gamma_n)|dm_n \colon \|x^*\| \leq 1\right\}< \varepsilon.\,
\end{equation}
If $\Gamma_n = \Gamma$ for every $n \in \mathbb{N}$ we obtain  the uniformly absolutely   continuous  scalar $(m_n)$-integrability
 of \, $\Gamma$ on $\Omega$.
\end{deff}

We recall that
the space $X$ is said to be {\it weakly compactly generated} (WCG) if it contains a weakly compact subset  that is linearly dense in $X$
(see for example \cite{M2}).
%\mg{\tiny Th4}
\begin{thm}\label{Th4}
 Let $\Gamma, \Gamma_n :\vO \rightarrow cwk(X)$, $n \in\N$,  be scalarly measurable multifunctions.
 Morover let $(m_n)_n, \ m$  be  measures. Assume that
 \begin{itemize}	
	\item[\rm \rm (\ref{Th4}.j)] the sequence  $ (\Gamma_n )_n$
		 has uniformly absolutely continuous scalar  $(m_n)$-integrals on $\vO$;
 \item[\rm \rm (\ref{Th4}.jj)]   $(m_n)_n$ setwise converges   to $m$;
	\item[\rm \rm (\ref{Th4}.jjj)]	each  multifunction $\Gamma_n$	is Pettis integrable with respect to $m_n$;
	\item[\rm \rm (\ref{Th4}.jv)]  $\Gamma$ is scalarly integrable with respect to $m$.
	\end{itemize}
	 If  for every $A\in \mathcal{A}$ and every $x^*\in{X^*}$
%\mg{\tiny em1}
 \begin{equation}\label{em1}
	\lim_n \int_{A} s(x^*, \Gamma_n)\, dm_n =  \int_{A}s (x^*, \Gamma)\, dm
\end{equation}
then $\vG$ is Pettis integrable in $cwk(X)$ with respect to $m$.
\end{thm}
 \begin{proof}
  At first we show that the operator
$T_{\Gamma}: X^* \to L^1(m)$, defined as $T_{\Gamma}(x^*)=s(x^*, \Gamma)$ is bounded.
Since the function $\Gamma$ is scalarly integrable with respect to $m$,
$\Gamma$ is Dunford-integrable
in $cw^*k(X^{**})$,  where in $X^{**}$ we consider the $w^*$-topology, and
for every $A \in \mathcal{A}$ there exists $M^D_{\Gamma} (A) \in cw^*k(X^{**})$ such that,  for every
$x^* \in X^*,$
\begin{equation} \label{D1}
s(x^*, M^D_{\Gamma} (A)) =\int_As(x^*,\Gamma) dm,
\end{equation}
see for example \cite[Theorem 3.2]{mimmo3}.\\
Moreover
 the set $\{s(x^*,\Gamma):\|x^*\|\leq1\}$ is bounded in $L^1(m)$.
 Indeed it follows by (\ref{D1})
\[\int_{\Omega} |s(x^*,\Gamma) |dm \leq
2 \sup_{A\in {\mathcal A} }\left| \int_A s(x^*,\Gamma) dm \right|=
 2 \sup_{A \in {\mathcal A} } |s(x^*,M^D_{\Gamma} (A)) | < \infty,\]
where the last equality follows from the fact that for each $x^* \in X^*$, $s(x^*,M^D_{\Gamma} ( \cdot))$ is a scalar measure.
Hence, by the Banach–Steinhaus Theorem the set $
\displaystyle{\bigcup_{A \in \mathcal{A}} }M_{\Gamma}^D(A) \subset X^{**}$ is bounded  and then
\[
\sup_{\| x^*\| \leq 1} \int_{\Omega} |s(x^*,\Gamma) |dm \leq 2 \sup \left\{ \| x \| : x \in
\bigcup_{A \in \mathcal{A}} M_{\Gamma}^D(A) \right\} < \infty.\]
Therefore  the operator $T_{\Gamma}$ is bounded. \\
Now fix $\varepsilon >0$ and $x^* \in B_{X^*}$.  By (\ref{Th4}.j) there exists
$\delta>0$ satisfying (\ref{formula-def3.1}).
Let $E \in  {\mathcal A}$ be such that $m(E) < \delta/2$ and  set
 $E^+=\{ t \in E : s(x^*, \Gamma(t)) \geq 0\}$ and
\mbox{$E^-=\{ t \in E : s(x^*, \Gamma(t)) <0\}$}.  From
 (\ref{em1}) and
(\ref{Th4}.jj) we find $ N_1 \geq N $ such that
   $m_n(E) < \delta$ for every $n \geq N_1$ and
\[ \int_{E^+} s(x^*, \Gamma)dm < \left| \int_{E^+} s(x^*, \Gamma_{N_1})dm_{N_1} \right| + \frac{\varepsilon}{2}\]
and
\[ \left|\int_{E^-} s(x^*, \Gamma)dm \right| < \left| \int_{E^-} s(x^*, \Gamma_{N_1})dm_{N_1} \right| + \frac{\varepsilon}{2}.\]
So,  by (\ref{Thmulti}.j), we get
\begin{eqnarray*}
\int_E |s(x^*, \Gamma)| dm &=& \int_{E^+} s(x^*, \Gamma)dm + \left|\int_{E^-} s(x^*, \Gamma)dm \right| \\
&<& \left| \int_{E^+} s(x^*, \Gamma_{N_1})dm_{N_1} \right| + \left| \int_{E^-} s(x^*, \Gamma_{N_1})dm_{N_1} \right| +\varepsilon\\
&\leq & \int_E |s(x^*, \Gamma_{N_1})| dm_{N_1} + \varepsilon < 2 \varepsilon
\end{eqnarray*}
\noindent  and the scalar uniform integrability  with respect to $m$ of $\Gamma$ follows.
Then the operator $T_{\Gamma}: X^* \to L^1(m)$ is weakly compact.\\
Now we are proving that $\Gamma$ is determined by a $WCG$ generated subspace of $X$.
Since, for each $n \in\N$, $\Gamma_n$ is Pettis integrable, by \cite[Teorem 2.5]{M2}, let $Y_n$ be a $WCG$ subspace of $X$ generated by a
weakly compact convex set $W_n \subset B_{X^*}$ and determining the multifunction $\Gamma_n$. The set $\sum 2^{-n} W_n$ is
a weakly compact set generating a space $Y$. We want to prove that  $\Gamma$ is determined by $Y$.\\
 Let $y^*\in{Y^{\perp}}$, let $\Omega^+=\{t  : s(y^*, \Gamma(t)) \geq 0\}$ and $A_n:=\{t \in \Omega^+ :s(y^*, \Gamma_n(t))=0\}$.
 Then $m_n(A_n)=m_n(\vO^+)$.
Let $A:=\limsup_n{A_n}=\bigcap_{k=1}^{\infty}\bigcup_{p=k}^{\infty}A_p$. Then
\begin{eqnarray*}
	m(A)&=&\lim_km(\bigcup_{p=k}^{\infty}A_n)\geq\limsup_km(A_k)\geq\limsup_km_k(A_k)\\
	&=&\limsup_km_k(\vO^+)=m(\vO^+).
\end{eqnarray*}
It follows by equality (\ref{em1}) that $s(y^*,\Gamma(t))=0$ $m$-a.e. on the set $\Omega^+$. Analogously if we denote by
 $\Omega^-=\{t  : s(y^*, \Gamma(t)) < 0\}$ it follows that  $s(y^*,\Gamma(t))=0$ $m$-a.e. on the set $\Omega^-$.
Thus, $Y$ determines the multifunction $\Gamma$ and the Pettis integrability of $\Gamma$ follows by \cite[Theorem 2.5]{M2}.
\end{proof}
 %%%%%% ar

As a consequence of Theorem \ref{Th4} we obtain
 \begin{thm}\label{Thmulti}
%\mg{\tiny Thmulti}
  Let $\Gamma, \Gamma_n :\vO \rightarrow cwk(X)$, $n \in\N$,  be scalarly measurable multifunctions.
 Morover let $(m_n)_n, \ m$  be  measures.
		Suppose that	
	\begin{itemize}	
	\item[\rm \rm(\ref{Thmulti}.j)] the sequence  $ (\Gamma_n )_n$
		 has uniformly absolutely continuous scalar  $(m_n)$-integrals on $\vO$;
	\item[\rm \rm (\ref{Thmulti}.jj)]  $s(x^*, \Gamma_n)  \rightarrow s(x^*, \Gamma)$,  in $m$-measure, for each $x^* \in X^*$;			
	%\item[\rm (\ref{Thmulti}.jjj)]  $\Gamma$ is  scalarly $m$-integrable;	
	\item[\rm \rm (\ref{Thmulti}.jjj)]  $\Gamma$ has uniformly absolutely continuous scalar  $(m_n)$-integrals on $\vO$;		
	\item[\rm \rm (\ref{Thmulti}.jv)]   $(m_n)_n$ setwise converges   to $m$;
	\item[\rm \rm (\ref{Thmulti}.v)]	each  multifunction $\Gamma_n$	is Pettis integrable with respect to $m_n$.	
	\end{itemize}
	
	Then the multifunction $\Gamma$ is Pettis integrable with respect to $m$ in $cwk(X)$ and
\[\lim_n s\left(x^*, \int_{A} \Gamma_n\, dm_n \right)=  s\left(x^*, \int_{A} \Gamma\, dm \right), \]

	for every $x^* \in X^*$ and for every $A \in {\mathcal A}$.
	\end{thm}
\begin{proof}
 First of all we observe that, by Theorem \ref{Th1}, we have  that for every $A \in  {\mathcal A}$	
%\mg{\tiny em11}
\begin{equation}\label{em11}
	\lim_n \int_{A} s(x^*, \Gamma_n)\, dm_n =  \int_{A}s (x^*, \Gamma)\, dm.
\end{equation}
Moreover by Proposition \ref{p1} $\Gamma$ is scalarly integrable, therefore the Pettis integrability of $\vG$ with respect to $m$ is a
 consequence of Theorem \ref{Th4}.
	\end{proof}

 The following definition is a generalization of the notion of the scalar equi-convergence in measure for a sequence of scalarly measurable
 multifunctions $(\Gamma_n)_n$ (see \cite[p.852]{mu8} and  \cite{bm} for the vector case).
%If $m_n=m$ for every $n \in \mathbb{N}$ we have the   scalarly equi-convergent in measure with respect to $m$.
\begin{deff}
Let $\vG, \vG_n\colon \vO\to{cwk(X)}$ be scalarly measurable multifunctions. We say that the sequence $(\vG_n)_n$
 is  scalarly equi-convergent in measure with respect to a sequence of measures $(m_n)_n$
%\subset \mathcal{M}^+(\Omega)$
to  $\vG$ if, for every
$\delta>0$,
%\mg{\tiny e80}
\begin{equation}\label{e80}
	\lim_n\sup_{\|x^*\|\leq1}m_n\{ t\in\vO\colon |s(x^*,\vG_n(t))-s(x^*,\vG(t))|>\delta\}=0.
\end{equation}
\end{deff}
\vskip.2cm
If in  the Theorem \ref{Thmulti} we substitute  the convergence in condition (\ref{Thmulti}.jj) with the scalar equi-convergence
 in measure and the setwise convergence of $m_n$ to $m$ with the convergence in total variation, we get a stronger result.

%%%%%%%%%

%%%%%%%%%%%
%\mg{\tiny Thmulti2}
\begin{thm}\label{Thmulti2}   Let $\Gamma, \Gamma_n :\vO \rightarrow cwk(X)$, $n \in\N$,  be scalarly measurable multifunctions.
 Morover let $  (m_n)_n, \ m$,  be  measures.
		Suppose that	
	\begin{itemize}	
	\item[\rm (\ref{Thmulti2}.j)] the sequence  $ (\Gamma_n )_n$
		 has uniformly absolutely continuous scalar  $(m_n)$-integrals on $\vO$;
	\item[\rm (\ref{Thmulti2}.jj))]  the sequence $(\vG_n)_n$, $n\in\N$, is scalarly equi-convergent in measure with respect to 
$(m_n)_n$ and $m$ to  $\vG$;			
	%\item[\rm (\ref{Thmulti}.jjj)]  $\Gamma$ is  scalarly $m$-integrable;	
	\item[\rm (\ref{Thmulti2}.jjj)]    $\Gamma $   has  uniformly absolutely   continuous scalar $(m_n)$ and $m$ integrals;
	 \item[\rm (\ref{Thmulti2}.jv)]$(m_n)_n$ is convergent to $m$ in total variation;
	\item[\rm (\ref{Thmulti2}.v)]	each  multifunction $\Gamma_n$	is Pettis integrable with respect to $m_n$.	
	\end{itemize}
	
	Then the multifunction $\Gamma$ is Pettis integrable with respect to $m$ in $cwk(X)$ and
\[ \lim_n d_H(M_{\Gamma_n}(A), M_{\Gamma}(A))=0\]
	uniformly in $A \in {\mathcal A}$,
where $M_{\Gamma_n}, \, (M_{\Gamma}): \mathcal{A} \to cwk(X)$ are  the $m_n$ $(m)$-Pettis integrals of the multifunction 
$\Gamma_n \ (\Gamma)$  respectively.
	\end{thm}

%===================

%======================================
\begin{proof}
 For every $x^* \in X^*$ and $A \in {\mathcal A}$  it is enough to apply  \cite[Theorem 2.8]{serfozo} to get
\begin{equation*}
	\lim_n \int_{A} s(x^*, \Gamma_n)\, dm_n =  \int_{A}s (x^*, \Gamma)\, dm
\end{equation*}
and as in Theorem \ref{Thmulti} the multifunction  $\Gamma$ is Pettis integrable with respect to $m$.
In order to prove the convergence in the Hausdorff metric, fix $A \in {\mathcal A}$,  $\varepsilon>0$ and  $\delta>0 $ that satisfy \eqref{e80} also in case of
 $m_1=m_2=\ldots={m}$ and $\Gamma_1=\Gamma_2=\ldots ={\Gamma}$.
Fix also $0<\eta<\ve$. For each $n\in\N$ and $x^*\in B_{X^*}$ denote by $H_{n,x^*}$ the set
\[H_{n,x^*}:= \{t \in\vO\colon |s(x^*,\vG_n(t))-s(x^*,\vG(t))|>\eta \}.\]
By the
assumption of the scalar equi-convergence in measure with respect to $m_n$ and $m$, there exists $k\in\N$ such that  for all $n\geq k$
\[\sup_{\|x^*\|\leq 1} \max \bigg\{ m_n(H_{n,x^*}),  m(H_{n,x^*}) \bigg \}<\delta.\]
 Then, for all $n\geq{k}$ and $\|x^*\|\leq 1$
\begin{eqnarray*}
	&&\sup_{\|x^*\| \leq 1} \biggl|\int_A s(x^*,\vG_n)\,dm_n-\int_A s(x^*,\vG)\,dm  \biggr|\leq   \\
	& \leq& \sup_{\|x^*\| \leq 1} \biggl|\int_{A\cap H_{n,x^*}}s(x^*,\vG_n)\,dm_n-\int_{A\cap H_{n,x^*}}s(x^*,\vG)\,dm\biggr|+
	\\ &+& \sup_{\|x^*\| \leq 1}
\biggl|\int_{A\cap H^c_{n,x^*}}s(x^*,\vG_n)\,dm_n-\int_{A\cap H^c_{n,x^*}}s(x^*,\vG)\,dm\biggr|.
\end{eqnarray*}

Observe that, by (\ref{Thmulti2}.j-\ref{Thmulti2}.jjj) and formula (\ref{e80}), we have
\begin{eqnarray*}
	&& \sup_{\|x^*\| \leq 1}\biggl|\int_{A  \cap H_{n,x^*}}s(x^*,\vG_n)\,dm_n-\int_{A\cap H_{n,x^*}}s(x^*,\vG)\,dm\biggr| \leq \\
	&&  \sup_{\|x^*\| \leq 1}\biggl|\int_{A \cap H_{n,x^*}}s(x^*,\vG_n)\,dm_n \biggr|+
 \sup_{\|x^*\| \leq 1}\biggl|\int_{A\cap H_{n,x^*} }s(x^*,\vG)\,dm\biggr| \leq 2\ve.
\end{eqnarray*}
Relatively to the second summand we apply Remark \ref{nota} and formula (\ref{variaz}) and we obtain
\begin{eqnarray*}
	&& \sup_{\|x^*\| \leq 1} \biggl|\int_{A\cap H^c_{n,x^*}} s(x^*,\vG_n)\, dm_n - \int_{A\cap H^c_{n,x^*}} s(x^*,\vG) \,dm \biggr| \leq \\
	&& \leq  \sup_{\|x^*\| \leq 1} \biggl|\int_{A\cap H^c_{n,x^*}}s(x^*,\vG_n)\,dm_n- \int_{A\cap H^c_{n,x^*}}s(x^*,\vG)\,dm_n\biggr| + \\
	&&+  \sup_{\|x^*\| \leq 1}  \biggl| \int_{A\cap H^c_{n,x^*}} s(x^*,\vG)\, dm_n - \int_{A\cap H^c_{n,x^*}} s(x^*,\vG) \,dm\biggr| \leq \\
	&&\leq \sup_{\|x^*\| \leq 1} \eta \, m_n(A\cap H^c_{n,x^*})
	 +  \sup_{\|x^*\| \leq 1} \biggl| \int_{A\cap H^c_{n,x^*}} s(x^*,\vG) \, d (m_n-m) \biggr| \leq
	\\ &&\leq  \varepsilon \sup_n m_n (\Omega) +  \sup_{\|x^*\| \leq 1}  \int_{A\cap H^c_{n,x^*}} |s(x^*,\vG)| d |m_n-m|.
\end{eqnarray*}
Let % $\ov{m}:= \sum_{k=1}^{\infty} \frac{|m_n-m|(E)}{2^n (1+ |m_n-m|(\Omega))}$
$\ov{m}:= \sum_{k=1}^{\infty} |m_n-m|(E) \cdot (2^n (1+ |m_n-m|(\Omega)))^{-1}$
 be a probability measure on $\mathcal{A}$ such that $|m_n-m|\ll\ov{m}$ for every $n\in\N$.
According to \cite[Proposition 1.2]{M2} there exists a measurable function
$\varphi_{\Gamma}:\vO\to[0,\infty)$ such that for each
$x^*\in{X^*}$ the inequality
$|s(x^*,\vG(t)) |\leq \varphi_{\Gamma}(t)\|x^*\|$ holds true $\ov{m}$-a.e.
 In particular the inequality holds true also $|m_n-m|$-a.e., for each $n\in\N$ separately.\\
Let $a>0$  and $F\in  \mathcal{A}$ be such that
$\varphi_{\Gamma}(t) \chi_F\leq{a}$ $\ov{m}$-a.e. and $m(F^c)<\delta$.
Since $(m_n)_n$  converges to $m$, there exists $\N\ni\tilde{k} \geq \ov{k}$
such that $m_n(F^c)<\delta$ for every $n\geq\tilde{k}$. Then, let $\N\ni\check{k}>\tilde{k}$
be such that $|m_n-m|(F)<\ve/a$ for every $n\geq\check{k}$. If $n\geq\check{k}$, then
\begin{eqnarray*} 
&& \sup_{\|x^*\| \leq 1}  \int_{\vO} |s(x^*,\vG) |\, d |m_n-m|\\
&\leq& \sup_{\|x^*\| \leq 1}  \int_F |s(x^*,\vG) |\, d |m_n-m|+
\sup_{\|x^*\| \leq 1}  \int_{F^c} |s(x^*,\vG) \, d |m_n-m|\\
&\leq& a|m_n-m|(F)+\sup_{\|x^*\| \leq 1}  \int_{F^c} |s(x^*,\vG) |\,dm+\sup_{\|x^*\| \leq 1}  \int_{F^c} |s(x^*,\vG)| \, d m_n
%\\&\leq&
\leq 3\ve.
\end{eqnarray*}
Then,
\begin{eqnarray*}
d_H(M_{\Gamma_n}(A), M_{\Gamma}(A))&=& \sup_{\|x^*\|\leq 1}|s(x^*, M_{\Gamma_n}(A)) -s(x^*,M_{\Gamma}(A))|  \\
&\leq &\sup_{\|x^*\|\leq 1}\biggl|\int_A s(x^*,\vG_n)\,dm_n-\int_A s(x^*,\vG)\,dm\biggr|\\
&\leq &3\ve+\ve\sup_nm_n(\vO).
\end{eqnarray*}
 Since $\ov{k}$  is indipendent to $A \in \mathcal{A}$, then the previous convergence is uniform with respect to $A$.
\end{proof}

%============capitolo vettoriale =====================================
%========================== ===========================================================
\subsection{The vector case for integrands}\label{tre}
We consider now  vector valued functions which are a particular case of  the multivalued one.
Therefore Theorem \ref{Th4} can be rewritten as follows:
\begin{thm}\label{caso-v}
 Let $f, f_n :\vO \rightarrow X$, $n \in\N$,  be scalarly measurable functions.
 Morover let $(m_n)_n, \ m$  be  measures. Assume that
 \begin{itemize}	
	\item[\rm (\ref{Th4}.j)] the sequence  $ (f_n )_n$
		 has uniformly absolutely continuous scalar  $(m_n)$-integrals on $\vO$;
 \item[\rm (\ref{Th4}.jj)]   $(m_n)_n$ setwise converges   to $m$;
	\item[\rm (\ref{Th4}.jjj)]	each  function $f_n$	is Pettis integrable with respect to $m_n$;
	\item[\rm (\ref{Th4}.jv)]  $f$ is scalarly integrable with respect to $m$.
	\end{itemize}
	 If  for every $A\in \mathcal{A}$ and every $x^*\in{X^*}$
%\mg{\tiny em1v}
 \begin{equation}\label{em1v}
	\lim_n \int_{A} x^*f_n\, dm_n =  \int_{A} x^*f\, dm
\end{equation}
then $f$ is Pettis integrable in $X$ with respect to $m$.
\end{thm}
The following theorem is the vector valued formulation of Theorem \ref{Thmulti}.
 We  present here a  proof of the Pettis integrability of the limit function $f$ using a characterization of weakly compact 
sets and without using the multivalued case.

%\mg{\tiny Th2v}
\begin{thm}\label{Th2v}
	Let $f,f_n:\vO \rightarrow X$ be scalarly measurable  functions.  Moreover let
 $(m_n)_n$ and $m$   be  measures.	
	Suppose that 	
	\begin{itemize}	
		\item[\rm (\ref{Th2v}.j)]  $(f_n)_n$ has  uniformly absolutely continuous scalar $(m_n)$- integrals on $\vO$;	
		\item[\rm (\ref{Th2v}.jj)]  $x^*f_n(t)  \rightarrow x^*f(t)$,    in $m$-measure, for each $x^* \in X^*$;		
		%\item[\rm (\ref{Th2v}.jjj)]  $f$ is scalarly $m$-integrable;
			\item[\rm (\ref{Th2v}.jjj)] $f$ has uniformly absolutely continuous  scalar $(m_n)$-integrals on $\vO$;			
		\item[\rm (\ref{Th2v}.jv)]  $(m_n)_n$ is setwise  convergent to $m$;
		\item[\rm (\ref{Th2v}.v)]	each  function $f_n$	is  Pettis integrable with respect to $m_n$.	
	\end{itemize}	
	Then $f$ is Pettis integrable  with respect  to $m$ and 	
	\[ \lim_n \int_{\vO} f_n dm_n =  \int_{\vO}f dm, \]
	weakly in $X$.
\end{thm}

\begin{proof}
Let $A \in {\mathcal A}$.  For every $x^* \in X^*$ it is enough to apply  Theorem \ref{Th1} to get
%\mg{\tiny Pettis8}
	\begin{equation}\label{Pettis8}
		 \lim_n \int_{A}x^* f_ndm_n =  \int_{A}x^*fdm.
	 \end{equation}	
Therefore to get our thesis it is enough to prove that $f$ is Pettis integrable. 	
By equality (\ref{Pettis8}) we have that the sequence $\left(\dint_{\vO}f_n dm_n \right)_n$ is weakly Cauchy.\\
By a Grothendieck characterization  of weakly compact sets (\cite{G}) and  taking into account hypothesis  (\ref{Th2v}.v),
% era scritto jv   vedii handbook pag 550
 it is enough to show that for each $(f_{n_j})_j$ and $(x^*_k)_k \subset B_{X^*}$
%\mg{\tiny Pettis9}
\begin{equation}\label{Pettis9}
		 \alpha:= \lim_k \lim_j \left<x^*_k,  \int_{\vO} f_{n_j}dm_{n_j}\right> =
		\beta:=\lim_j \lim_k \left<x^*_k,  \int_{\vO}f_{n_j}dm_{n_j}\right>,
\end{equation}
provided all above limits exist. \\
Since the function $f$ is scalarly integrable with respect to $m$,
$f$ is Dunford-integrable and the set $\{x^*f:\|x^*\|\leq1\}$ is bounded in $L^1(m)$.
 To see it let $\nu$ be the Dunford integral of $f$. Then
the set $\{\nu(E):E\in\mathcal{A}\}$ being the range of $\mathcal{A}$ in $X^{**}$
 is a bounded set (because it is weak$^*$ bounded).\\
If $\pi=\{E_{1},...,E_{n}\}$ is a partition of $\vO$ into  pairwise
disjoint members of $\mathcal{A}$,
 and $x^*\in B_{X^*}$, then
\begin{eqnarray*}
	\sum_{E_i\in\pi}|x^*\nu(E_i)|&=&\sum_{E_i\in\pi^+}x^*\nu(E_i)
	-\sum_{E_i\in\pi^-}x^*\nu(E_i)=\\
	&=&x^*\{\sum_{E_i\in\pi^+}\nu(E_i)\}-
	x^*\{\sum_{E_i\in\pi^-}\nu(E_{i})\}\le\\
	&\leq&2\sup\{\|\nu(E)\|:\,E\in\vS\}
\end{eqnarray*}
 where
$\pi^+= \{E_i : x^*\nu (E_i) \geq 0\}$ and $\pi^-= \{E_i : x^*\nu (E_i) <0\}$,
Hence, if $x^*\in{B_{X^*}}$, then
%\mg{\tiny e1}
\begin{equation}\label{e1}
\int_{\vO}|x^*f|\,dm=|x^*\nu|(\vO)\leq 2\sup\{\|\nu(E)\|:\,E\in\mathcal{A} \}<\infty.
\end{equation}	
At first we are   proving  that the sequence $(x^*_k f)_k$   has  uniformly absolutely
	continuous   integrals  with respect to $m$ on $\vO$. \\	
	Now let $E \in  {\mathcal A}$ be such that $m(E) < \delta/2$.  Set
	$E_k^+=\{t \in E: x^*_k f(t)>0\}$ and \mbox{$E_k^-=\{t \in E: x^*_k f(t)\leq 0\}$}.
	Moreover let $N_k(\varepsilon) \in\N$ be such that for all $n >N_k(\varepsilon)$
%\mg{\tiny Pettis11}
\begin{eqnarray}\label{Pettis11}
\sup \left\{ \left|\int_{E_k^+}x^*_k f_n dm_n - \int_{E_k^+}x^*_k f dm \right|,
\,
 \left|\int_{E_k^-}x^*_k f_n dm_n - \int_{E_k^-}x^*_k f dm \right| \right\} < \varepsilon. \quad
	\end{eqnarray}	
	Let $n_k >N_k(\varepsilon)$ be such that $m_{n_k}(E) \leq m(E) + \delta/2 < \delta$.	
	Therefore, taking into account hypothesis (\ref{Th2v}.j),  by (\ref{Pettis11})  we have 	
\begin{eqnarray*}
		 \int_{E} |x^*_kf| dm &=& \left| \int_{E_k^+} x^*_kf dm \right| +   \left| \int_{E_k^-} x^*_kf dm \right| \\
		  &\leq&	\left|\int_{E_k^+}x^*_k f_{n_0}dm_{n_0} - \int_{E_k^+}x^*_k fdm \right|  +\\
		&+&
		 \left|\int_{E_k^-}x^*_k f_{n_0}dm_{n_0} - \int_{E_k^-}x^*_k fdm \right|\\
		&+&\ \int_{{E_k^+}} |x^*_k f_{n_0}| d{m_{n_0}} + \ \int_{{E_k^-}} |x^*_k f_{n_0}| d{m_{n_0}}
		< 4 \varepsilon.
\end{eqnarray*}
	Since the sequence $(x^*_k f)_k$  has  uniformly absolutely   continuous
  integrals  with respect to $m$ on $\vO$  and it is bounded in $L^1(m)$, it is weakly relatively compact.
	 So we have  the existence of a subsequence  $(z^*_k)_k$ of $(x^*_k )_k$ and of a real valued function $g \in L^1(m)$ such that
	$z^*_kf \rightarrow g$ weakly in $L^1(m)$. Mazur's Theorem yields  the existence of functionals
	$w^*_k \in co\{z^*_j: \ j \geq k\}$
	such that
	\[\lim_k \int_{\vO}|w^*_k f -g|dm =0 \  \qquad\mbox{and} \  \lim_k w^*_k f =g,  \ \qquad\mbox{m-a.e.}\,.\]
	
	If $w^*_0$ is a weak$^*$-closter point of  $(w^*_k)_k$, then $g=w^*_0f$  $m$-a.e..	
	Therefore $\alpha= \int_{\vO} w_0^*fdm.$	
	On the other hand	
\begin{eqnarray*}
			\lim_k \left<x^*_k,  \int_{\vO}f_{n_j}dm_{n_j}\right>&=&
			\lim_k \left<w^*_k,  \int_{\vO}f_{n_j}dm_{n_j}\right>
			=\left<w^*_0,  \int_{\vO}f_{n_j}dm_{n_j}\right>\\ &=&
			 \int_{\vO}w^*_0f_{n_j}dm_{n_j}.
\end{eqnarray*}	

 By the hypothesis  (\ref{Th2v}.j-\ref{Th2v}.jv)  and by Theorem \ref{Th1} 	it follows	
	\[ \beta:=\lim_j \int_{\vO}w^*_0f_{n_j}dm_{n_j}= \int_{\vO}w^*_0fdm.\]
	So $\alpha=\beta$ and this completes the proof.
\end{proof}

In particular
%\mg{\tiny Th3}
\begin{cor}\label{Th3}
Let $(m_n)_n$ be a sequence of measures that is setwise convergent to  a measure $m$.
 If $f:\vO\to{X}$ has uniformly absolutely continuous
 scalar $(m_n)$-integrals on $\vO$, then $f$ is scalarly $m$-integrable.
 If moreover  $f$	is  Pettis integrable with respect to $m_n$ for each $n \in \mathbb{N}$, then $f$ is
Pettis integrable with respect to $m$
and for every $A \in \mathcal{A} $ it is
	\[ \lim_n \int_{A} f dm_n =  \int_{A}f dm, \]
	weakly.
\end{cor}
\begin{proof}
The first assertion follows from Proposition \ref{p1}.
The Pettis integrability of $f$ with respect to $m$ follows from Theorem \ref{Th2v} when $f_n=f$ for every $n \in \mathbb{N}$.
\end{proof}

As in Theorem \ref{Thmulti2}, if we consider
the scalarly equi-convergence in measure
 and
the convergence in total variation,
we have
\begin{thm}\label{Th1m}
%\mg{\tiny Th1m}
  Let $f, f_n: \vO \to X$  be  measurable functions and let
 $(m_n)_n$,  $m$ be  measures.
	 Suppose that
	 \begin{itemize} 	
	 	\item[\rm (\ref{Th1m}.i)]  $(f_n)_n$  has uniformly absolutely   continuous  scalar $(m_n)$-integrals   on $\vO$; 	
	 	\item[\rm (\ref{Th1m}.ii)]  $(f_n)_n$ is  scalarly equi-convergent in measure to $f$ with respect to $(m_n)_n$ and $m$;
	 		\item[\rm (\ref{Th1m}.iii)]   $f $   has  uniformly absolutely   continuous scalar $(m_n)$ and $m$ integrals;
	 	\item[\rm (\ref{Th1m}.iv)] $(m_n)_n$ is convergent to $m$ in total variation;
\item[\rm (\ref{Th1m}.v)]	each  function $f_n$ is  Pettis integrable with respect to $m_n$.
	 \end{itemize} 	
 	 Then,  $f$ is Pettis-integrable with respect to $m$ and
	\begin{equation*}
		 \lim_n \, \biggl\|\int_{A} f_n\,dm_n - \int_{A} f\,dm\biggr\|=0\,\,
	\end{equation*}
uniformly with respect to $\, A \in \mathcal{A}$.
\end{thm}

%============================
%============================
%============================
%============================
%============================

%============================

\subsection{The vector case for McShane integrable integrands }\label{Mcshane}

Now we are going to examine the behavior of  McShane integrable  integrands.
 The McShane integral is  a ``gauge defined integral'' and, also in the case of the generalized
  McShane integral introduced in 1995  by D. Fremlin in a measure space  $\vO$, we  need a topology  in $\vO$.
 Briefly we recall the definition (see \cite{Frem-IJM}). For simplicity we prefer to use finite partitions. 
So our framework $\Omega$ is a compact Radon measure space.\\
Let $(\vO,  \mathcal{A})$ be a compact Radon measure space space with measure $m$ and topology $\Tau$.
 A {\it finite strict generalized McShane partition} of $\vO$ is a family $\{(A_i, t_i)\}_{i \leq p}$ such that $A_1,...,A_p$
 is a finite disjoint cover of $\vO$ by elements of $\mathcal A$ and $t_i \in \Omega, i=1, \ldots, p$. 
A {\it gauge} $\Delta$ on $\vO$ is a function
$\Delta:\vO \to \Tau$ such that $t \in \Delta(t)$ for every $t \in \vO$. We say that  a (finite strict
 generalized) McShane partition $\{(A_i, t_i)\}_{i \leq p}$ is {\it subordinated}
 to a gauge $\Delta(t)$ if $A_i \subset \Delta(t_i)$ for $i=1,...,p$.\\
A function   $f:\vO \to X$ is said to be $m$-{\it McShane} ($m$-(MS))
integrable on $\vO$ with $m$-(MS)-integral $w \in X$ if for every $\varepsilon > 0$ there exists a gauge $\Delta$
such that for each
partition $\{(A_i, t_i)\}_{i \leq p}$  subordinated to $\Delta$,
we have
%\mg{\tiny Mc100}
\begin{eqnarray}\label{Mc100}
	\left\|\sum_{i=1}^p f(t_i)m (A_i)-w\right\|<\varepsilon.
\end{eqnarray}
We set $w:= (MS)\int_{\vO} fdm$.

\begin{deff}\label{mn-equi}
%\mg{\tiny mn-equi}
	\rm Let  $m_n$, $n=1,2,...$ be  measures.
	We say that a  sequence of $m_n$-(MS)integrable functions $f_n: \Omega \to X$
	is {\it $ (m_n)$-equi-integrable   on $\Omega $},  if for every $ \varepsilon  >0$ there exists a gauge $\Delta$ such that 	
	for every $n \in \mathbb{N}$			
	%\mg{\tiny Mc1010}
	\begin{equation}
\label{Mc1010}
\left\| \sum_{i=1}^pf_n(t_i) m_n(A_i) - (MS)\int_{\Omega}f_n dm_n \right\|<\varepsilon\,
\end{equation}	
for each  partition
$\{(A_i,t_i)\}_{i \leq p}$
subordinated to $\Delta$.
\end{deff}		

If $m_n=m$ for all $n\in\N$, then we have the classical condition of equi-integrability.\\

\begin{thm}\label{ThMcSequi}
	%\mg{\tiny ThMcSequi}
Let $m$ and $(m_n)_n$  be  measures, let $f_n:\Omega  \to X$ be $m_n-(MS)$-integrable  functions, $n \in \mathbb{N}$, and let $f:\Omega  \to X$.  	
	Suppose that
	\begin{itemize} 	
		\item[\rm (\ref{ThMcSequi}.i)]   the sequence  $(f_n)_n$ is $ (m_n)$-equi-integrable   on $\Omega $;
		\item[\rm (\ref{ThMcSequi}.ii)]  $f_n(t)  \rightarrow f(t)$,   for all $t \in \Omega $;
		\item[\rm (\ref{ThMcSequi}.iii)] $(m_n)_n$ is setwise  convergent to $m$.	
	\end{itemize} 	
	Then, $f$ is $m$-(MS)integrable and  for all $A \in {\mathcal A}$, 	
%\mg{\tiny Mc1020}
	\begin{equation}\label{Mc1020}
		\lim_n (MS)\int_{A} f_ndm_n =  (MS)\int_{A} fdm.
	\end{equation}	
Moreover if we substitute condition  {\rm (\ref{ThMcSequi}.iii)} with the convergence in total variation ($m_n\stackrel{tv}{\to}m$), then 
{ \rm (\ref{Mc1020})} holds uniformly in $A \in {\mathcal A}$.
\end{thm}

\begin{proof} %It is enough to prove that (\ref{Mc1020}) is true for $A=\Omega $.
Let $A \in \mathcal{A}$ be fixed.
If $\Delta $ is the gauge on $\Omega$ satisfying (\ref{ThMcSequi}.i) corresponding to the value $\varepsilon >0$, then for any $n \in \mathbb N$
%\mg{\tiny Mc1030}
\begin{equation} \label{Mc1030}
\left\| \sum_{i=1}^p f_n(t_i) m_n(A_i \cap A) - (MS)\int_{A}f_n dm_n \right\|
<\varepsilon
\end{equation}
for every
partition $\{(A_i, t_i) \}_{i \leq p}$ of $\Omega$
subordinated to $\Delta$ (see \cite[Theorem 1N]{Frem-IJM}).
Since the partition is fixed the pointwise convergence of $f_n$ to $f$ and the setwise convergence of $m_n$ to $m$ imply
that
\begin{eqnarray}\label{perA}
\lim_{n \to \infty}\sum_{i=1}^pf_n(t_i)m_n(A_i \cap A)=\sum_{i=1}^pf(t_i)m(A_i \cap A).
\end{eqnarray}
Choose $n_0 \in \mathbb N$ so that if $n,s >n_0$
\[\left\|\sum_{i=1}^pf_n(t_i)m_n(A_i \cap A) - \sum_{i=1}^pf_s(t_i)m_s(A_i \cap A) \right\|< \varepsilon.\]
Then we have
\begin{eqnarray*}
&&\left\| (MS)\int_{A} f_n dm_n -  (MS)\int_{A}f_sdm_s \right\|\\
&\leq&\left\| (MS)\int_{A}f_n dm_n -  \sum_{i=1}^pf_n(t_i)m_n(A_i \cap A) \right\| +\\
&+&
 \left\|\sum_{i=1}^p f_n(t_i)m_n(A_i) - \sum_{i=1}^p f_s(t_i)m_s(A_i \cap A) \right\|+\\
&+&\left\| \sum_{i=1}^p f_ s(t_i)m_s(A_i \cap A) -
(MS)\int_{A} f_sdm_s \right\| <3 \varepsilon,
\end{eqnarray*}
which shows that the sequence $\left((MS)\int_{A} f_n dm_n \right)_n$ is Cauchy, therefore it converges to $x_A \in X$.
%??Now let $\Delta$ satisfying   (\ref{ThMcSequi}.i) and (\ref{Mc1030}) and let $\{(A_1,t_1), \dots,(A_p,t_p) \}$ be a partition of $A$  subordinated to $\Delta$.
Then we have 
\begin{eqnarray*}
&&\left\|\sum_{i=1}^p f(t_i)m(A_i \cap A) -x_A \right\| \leq
\left\|\sum_{i=1}^p f(t_i)m(A_i \cap A) - \sum_{i=1}^p f_n(t_i) m_n(A_i \cap A) \right\| \\
&+& \left\| \sum_{i=1}^p f_ n(t_i)m_n(A_i \cap A) -(MS)\int_{A} f_n dm_n \right\| +\left\|(MS)\int_{A} f_n dm_n  -x_A \right\| < 3 \varepsilon.
\end{eqnarray*}
Therefore it follows that $f$ is $m$-(MS)integrable on $A$ and
\[\lim_n (MS)\int_{A} f_ndm_n =  (MS)\int_{A} fdm.\]
Finally,
if $m_n\stackrel{tv}{\to}m$, $n_0$ does not depend on $A$. Then formula (\ref{perA})  holds uniformly on ${\mathcal A}$
%\[
%\lim_{n \to \infty}\sum_{i=1}^pf_n(t_i)m_n(A_i\cap{A})=\sum_{i=1}^pf(t_i)m(A_i\cap{A}).
%\]
and the convergence in formula (\ref{Mc1020}) is uniform.
\end{proof}
%============================
We now  can consider the variety of multifunctions, including not only those
with weakly compact values but also those with bounded closed convex  values. We remember that
\begin{deff}
A multifunction  $\Gamma:\Omega \to cb(X)$ is said to be  {\it $m$-McShane}
   integrable on $\Omega$,  if there exists   $\varphi_{\Gamma} \in cb(X)$
    such that for every $\varepsilon > 0$ there exists a gauge $\Delta$ on $\Omega$
such that for each   partition
   $\{(A_i,t_i)\}_{i \leq p}$ of
   $\Omega$ subordinated to $\Delta$, we have
\begin{eqnarray}\label{e14}
d_H \left(\varphi_{\Gamma},\sum_{i=1}^p\Gamma(t_i)m(A_i)\right)<\ve\,.
\end{eqnarray}
\end{deff}
Since the McShane integral is  ``gauge defined", by a R{\aa}dstr\"{o}m embedding  it is immediate to extend the
 previous  Theorem \ref{ThMcSequi} for vector valued functions to the  McShane  integrable multifunctions.
A   R{\aa}dstr\"{o}m embedding theorem
states that the nonempty closed convex subsets of a Banach  space $X$ can be identified with points of
 $\ell_{\infty}(B_{X^*})$ such that the embedding map $i:
cb(X) \to \ell_{\infty}(B_{X^*})$
is additive, positively homogeneous, and isometric (see for example \cite{L1, DM2014,CASCALES2,mimmo0}).
This allows us to reduce the McShane integrability of multifunctions
to the McShane integrability of functions by embedding
$cb(X) \hookrightarrow \ell_{\infty}(B_{X^*})$.
%=============================
%===========================
The key point is that $i(cb(X))$ is a closed
 cone, consequently,  if $z\in \ell_{\infty}(B_{X^*})$ is the value of the  integral of $i\circ\vG$,
 then there exists a set $I\in{cb(X)}$ with $i(I)=z$.\\

In the following theorem the notion of $ (m_n)$-equi-integrability for multifunctions is analogous to that we have for functions
(Definition \ref{mn-equi}).

\begin{thm}\label{ThMc}
%\mg {\tiny it works }
Let $m$ and $(m_n)_n$   be  measures, let $ \vG_n:\Omega  \to cb(X)$  be $m_n-(MS)$-integrable  multifunctions, 
$n \in \mathbb{N}$, and let $ \vG:\Omega  \to cb(X)$.  	
Suppose that
	\begin{itemize} 	
		\item[\rm (\ref{ThMc}.i)]   the sequence  $(\vG_n)_n$ is $ (m_n)$-equi-integrable   on $\Omega $;
		\item[\rm (\ref{ThMc}.ii)]  $\lim_{n \to \infty} d_H(\vG_n(t),\vG(t)) = 0$,   for each $t \in \Omega $;
		\item[\rm (\ref{ThMc}.iii)] $(m_n)_n$ is setwise  convergent to $m$.	
	\end{itemize} 	
	Then, $\vG$ is $m$-(MS)integrable and  for all $A \in {\mathcal A}$, 	
	\begin{equation}\label{Mc1025}
		\lim_n (MS)\int_{A} \vG_n\,dm_n =  (MS)\int_{A} \vG\,dm.
	\end{equation}	
Moreover if we substitute condition  {\rm (\ref{ThMc}.iii)} with the convergence in total variation ($m_n\stackrel{tv}{\to}m$), then 
{ \rm (\ref{Mc1025})} holds uniformly in $A \in {\mathcal A}$.
\end{thm}
\begin{proof} We apply Theorem \ref{ThMcSequi} and the  R{\aa}dstr\"{o}m  embedding since
\begin{eqnarray*}
&&\lim_{n \to \infty}  i \circ \left( (MS) \int_A \Gamma_n \, dm_n \right) = \lim_{n \to \infty}
(MS) \int_A i \circ \Gamma_n \, dm_n = \\
&& (MS) \int_A i \circ \Gamma dm =
 i \circ \left( (MS) \int_A \Gamma\,  dm \right).
\end{eqnarray*}
\end{proof}

%Perhaps convergence in variation will give uniform convergence on $\mathcal A$?
%\\
{\bf \noindent Conclusions} Convergence results for varying measures are obtained both in weak and strong sense making use of the setwise
 convergence and the convergence in total variation respectively. By means of the Pettis integrability we are able 
to obtain the vector case as a particular case of the multivalued one. When we consider the McShane integrability
 we are able to pass from the vector case to the multivalued one.
 \\

{\bf \noindent Author's contribution}
All  authors  have  contributed  equally  to  this  work  for  writing,  review  and  editing. All authors have read and agreed
 to the published version of the manuscript.\\

%{\bf \noindent Conflict of interest} The authors declare no conflict of interest.\\

%{\bf \noindent Copyright } Not applicable\\

{\bf \noindent Funding}
This research has been accomplished within the UMI Group TAA “Approximation Theory and Applications”, the  G.N.AM.P.A. of INDAM and the  the
 Universities of Perugia and Palermo. \\ 

It was  supported by
by Ricerca di Base 2018 dell'Universit\`a degli Studi di Perugia - "Metodi di Teoria dell'Approssimazione, Analisi Reale, 
Analisi Nonlineare e loro Applicazioni";
 Ricerca di Base 2019 dell'Universit\`a degli Studi di Perugia - "Integrazione, Approssimazione, Analisi Nonlineare e loro Applicazioni"; 
"Metodi e processi innovativi per lo sviluppo di una banca di immagini mediche per fini diagnostici" funded by the Fondazione 
Cassa di Risparmio di Perugia (FCRP), 2018; "Metodiche di Imaging non invasivo mediante angiografia OCT 
sequenziale per lo studio delle Retinopatie degenerative dell'Anziano (M.I.R.A.)", funded by FCRP, 2019; 
 F.F.R. 2021 dell'Universit\`a degli Studi di Palermo: Marraffa.\\

% References: internal bibliography
%===================================================

%==========================
\Addresses

\vskip1cm 

\noindent LICENCE: CC BY-NC-ND \\
\href{https://doi.org/10.1016/j.jmaa.2022.126782}{https://doi.org/10.1016/j.jmaa.2022.126782}

\end{document}